\documentclass[12pt]{amsproc}
\usepackage{xy, amssymb, color, graphicx, tikz}
\xyoption{all}

\theoremstyle{plain}
\newtheorem{theorem}[equation]{Theorem}
\newtheorem*{theorem*}{Theorem}
\newtheorem*{conjecture*}{Conjecture}

\newtheorem*{prop*}{Proposition}
\newtheorem{lemma}[equation]{Lemma}
\newtheorem*{lemma*}{Lemma}
\newtheorem{cor}[equation]{Corollary}
\newtheorem*{cor*}{Corollary}

\newtheorem*{fq*}{Main question}

\theoremstyle{definition}

\newtheorem*{defn*}{Definition}
\newtheorem{remark}[equation]{Remark}
\newtheorem*{remark*}{Remark}
\newtheorem{example}[equation]{Example}
\newtheorem*{notation*}{Notation}

\newcommand{\bsl}{\backslash}

\newcommand{\Fq}{\mathbf F_q}

\newcommand{\id}[1]{\mathrm{id}_{#1}}
\newcommand{\inv}{^{-1}}
\newcommand{\irr}{\mathrm{Irr}}

\newcommand{\mat}[4]{
    \begin{pmatrix}
      #1 & #2 \\
      #3 & #4
    \end{pmatrix}
}

\renewcommand{\P}{\mathbf P}

\newcommand{\Z}{\mathbf Z}

\DeclareMathOperator{\Aut}{Aut}
\DeclareMathOperator{\End}{End}

\DeclareMathOperator{\Hom}{Hom}

\title{Orbits of pairs in abelian groups}
\author{C. P. Anilkumar}
\author{Amritanshu Prasad}
\address{The Institute of Mathematical Sciences, Chennai.}
\begin{document}
\maketitle
\begin{abstract}
  We compute the number of orbits of pairs in a finitely generated torsion module (more generally, a module of bounded order) over a discrete valuation ring.
  The answer is found to be a polynomial in the cardinality of the residue field whose coefficients are integers which depend only on the elementary divisors of the module, and not on the ring in question.
  The coefficients of these polynomials are conjectured to be non-negative integers.
\end{abstract}
\section{Introduction}
Let $R$ be a discrete valuation ring with maximal ideal $P$ generated by a unifomizing element $\pi$ and residue field $k=R/P$.
An $R$-module $M$ is said to be of bounded order if $P^NM=0$ for some positive integer $N$.
Let $\Lambda$ denote the set of all sequences of the form 
\begin{equation}
  \label{eq:lambda}
  \lambda = (\lambda_1^{m_1},\lambda_2^{m_2},\dotsc,\lambda_l^{m_l}),
\end{equation}
where $\lambda_1>\lambda_2>\dotsc>\lambda_l$ is a strictly decreasing sequence of positive integers and $m_1,m_2,\dotsc,m_l$ are non-zero cardinal numbers.
We allow the case where $l=0$, resulting in the empty sequence, which we denote by $\emptyset$.
Every $R$-module of bounded order is, up to isomorphism, of the form
\begin{equation}
  \label{eq:module-form}
  M_\lambda = (R/P^{\lambda_1})^{\oplus m_1}\oplus\dotsc \oplus (R/P^{\lambda_l})^{\oplus m_l} 
\end{equation}
for a unique $\lambda\in \Lambda$.
We will at times, wish to restrict ourselves to those $\lambda\in \Lambda$ for which all the cardinals $m_1,m_2,\dotsc,m_l$ are finite. We denote by $\Lambda_0$ this subset of $\Lambda$, which is the set of all partitions.
The $R$-module $M_\lambda$ is of finite length if and only if $\lambda\in \Lambda_0$.

Fix $\lambda\in \Lambda$, and write $M$ for $M_\lambda$.
Let $G$ denote the group of $R$-module automorphisms of $M$.
Then $G$ acts of $M^n$ by the diagonal action:
\begin{equation*}
  g\cdot (x_1,\dotsc,x_n) = (g(x_1),\dotsc,g(x_n)) \text{ for $x_i\in M$ and $g\in G$.}
\end{equation*}
For $n=1$, this is just the action on $M$ of its automorphism group.
A description of the orbits for this group action has been available for more than a hundred years (see Miller \cite{1905}, Birkhoff \cite{GarrettBirkhoff01011935}, and Dutta-Prasad \cite{MR2793603}).
Some qualitative results concerning $G$-orbits in $M^n$ for general $n$ were obtained by Calvert, Dutta and Prasad in \cite{calvert2013degeneration}.
In this paper we describe the set of $G$-orbits in $M^n$ under the above action for $n=2$.

This general set-up includes two important special cases, namely, finite abelian $p$-groups and finite dimensional primary $K[t]$-modules (isomorphism classes of which correspond to similarity classes of matrices).
The case of finite abelian $p$-groups arises when $R$ is the ring of $p$-adic integers and $\lambda\in \Lambda_0$.
The case of finite dimensional primary $K[t]$-modules arises when $R$ is the ring $E[[t]]$ of formal power series with coefficients in $E=K[t]/p(t)$ for some irreducible polynomial $p(t)\in K[t]$, and $\lambda\in \Lambda_0$.
The exact interpretation of this problem in terms of linear algebra is explained in Section~\ref{sec:relat-repr-quiv}. Specifically, the identities (\ref{eq:Rn1}) and (\ref{eq:gen-func}) relate the numbers of $G$-orbits in $M\times M$ with the problem of counting the number of isomorphism classes of representations of a certain quiver with certain dimension vectors.

Our key result (Theorem~\ref{theorem:sub-orbit}) is a description of the $G$-orbit of a pair in $M\times M$.
From this, when $k$ is finite of order $q$ and $\lambda\in \Lambda_0$, we are able to show that the cardinality of each orbit is a monic polynomial in $q$ (Theorem~\ref{theorem:cardinality-of-orbit}) with integer coefficients which do not depend on $R$.
Moreover, the number of orbits of a given cardinality is also a monic polynomial in $q$ with integer coefficients which do not depend on $R$ (Theorem~\ref{theorem:G_I-orbits-cardinalities}).
Theorem~\ref{theorem:G_I-orbits-cardinalities} gives an algorithm for computing the number of $G$-orbits in $M\times M$ of a given cardinality as a formal polynomial in $q$.
In particular, we obtain an algorithm for computing, for each $\lambda\in \Lambda_0$, a polynomial $n_\lambda(t)\in \Z[t]$ such that $n_\lambda(q)$ is the number of $G$-orbits in $M\times M$ is whenever $R$ has residue field of order $q$.
By implementing this algorithm in sage we have computed $n_\lambda(q)$ for all partitions $\lambda$ of integers up to $19$ at the time of writing. A sample of results obtained is given in Table~\ref{tab:nla}.
The sage program and a list of all $n_\lambda(q)$ are available from the web page \url{http://www.imsc.res.in/~amri/pairs/}.
\begin{table}
  \begin{center}
    \begin{tabular}{|c|c|}
      \hline
      \multicolumn{2}{|c|}{$\mathbf{n= 1 }$}\\
      \hline
      $ (1) $ & $ q + 2 $\\
      \hline
      \multicolumn{2}{|c|}{$\mathbf{n= 2 }$}\\
      \hline
      $ (2) $ & $ q^2 + 2q + 2 $\\
      $ (1, 1) $ & $ q + 3 $\\
      \hline
      \multicolumn{2}{|c|}{$\mathbf{n= 3 }$}\\
      \hline
      $ (3) $ & $ q^3 + 2q^2 + 2q + 2 $\\
      $ (2, 1) $ & $ q^2 + 5q + 5 $\\
      $ (1, 1, 1) $ & $ q + 3 $\\
      \hline
      \multicolumn{2}{|c|}{$\mathbf{n= 4 }$}\\
      \hline
      $ (4) $ & $ q^4 + 2q^3 + 2q^2 + 2q + 2 $\\
      $ (3, 1) $ & $ q^3 + 5q^2 + 7q + 4 $\\
      $ (2, 2) $ & $ q^2 + 3q + 5 $\\
      $ (2, 1, 1) $ & $ q^2 + 5q + 6 $\\
      $ (1, 1, 1, 1) $ & $ q + 3 $\\
      \hline
      \multicolumn{2}{|c|}{$\mathbf{n= 5 }$}\\
      \hline
      $ (5) $ & $ q^5 + 2q^4 + 2q^3 + 2q^2 + 2q + 2 $\\
      $ (4, 1) $ & $ q^4 + 5q^3 + 7q^2 + 6q + 4 $\\
      $ (3, 2) $ & $ q^3 + 5q^2 + 10q + 7 $\\
      $ (3, 1, 1) $ & $ q^3 + 5q^2 + 8q + 6 $\\
      $ (2, 2, 1) $ & $ q^2 + 6q + 8 $\\
      $ (2, 1, 1, 1) $ & $ q^2 + 5q + 6 $\\
      $ (1, 1, 1, 1, 1) $ & $ q + 3 $\\
      \hline
    \end{tabular}
  \end{center}
  \caption{The polynomials $n_\lambda(t)$}
  \label{tab:nla}
\end{table}
In general, we are able to show (Theorem~\ref{theorem:degree}) that $n_\lambda(q)$ is a monic polynomial with integer coefficients of degree $\lambda_1$ (the largest part of $\lambda$).
Our data leads us to make the following conjecture:
\begin{conjecture*}
  The polynomial $n_\lambda(q)$, which represents the number of $G$-orbits in $M\times M$ when $R$ has residue field of order $q$, has non-negative coefficients.
\end{conjecture*}

We are able to refine the results described above: the total number of $G$-orbits in $M\times M$ can be broken up into the sum of $G$-orbits in $A\times B$, as $A$ and $B$ run over $G$-orbits in $M$.
The parametrization of $G$-orbits in $M$ is purely combinatorial and does not depend on $R$, or even on $q$ (see Dutta-Prasad \cite{MR2793603}).
The orbits are parametrized by a certain set $J(P)_\lambda$ of order ideals in a lattice (see Section~\ref{sec:orbits-elements} for details).
For $I\in J(P)_\lambda$, let $M^*_I$ denote the orbit in $M$ parametrized by $I$.
For each pair $I,J\in J(P)_\lambda$, we are able to show (Theorem~\ref{theorem:main2}) that the number of $G$-orbits in $M^*_I\times M^*_J$ of any given cardinality is a polynomial in $q$ with integer coefficients which do not depend on $R$.

Our analysis of stabilizers in $G$ of elements of $M$ allows us to show that, for any $\lambda\in \Lambda$, the number of $G$-orbits in $M\times M$ does not change when $\lambda\in\Lambda$ is replaced by the partition derived from $\lambda$ by reducing each of the multiplicities $m_i$ of (\ref{eq:lambda}) to $\min(m_i,2)$ (Corollary~\ref{cor:independence of multiplicities}).
Thus, our calculations for the number of $G$-orbits in $M^*_I\times M^*_J$ extend to all $\lambda\in \Lambda$.
\section{Orbits of elements}
\label{sec:orbits-elements}
The $G$-orbits in $M$ have been understood quite well for over a hundred years (see Miller~\cite{1905}, Birkhoff\cite{GarrettBirkhoff01011935}, for $\lambda\in \Lambda_0$, and relatively recent work by Schwachh\"ofer and Stroppel~\cite{MR1656579}, for general $\lambda\in \Lambda$).
For the present purposes, however, the combinatorial description of orbits due to Dutta and Prasad~\cite{MR2793603} is more relevant.
This section will be a quick recapitulation of those results.

It turns out that for any module $M$ of the form (\ref{eq:module-form}), the $G$-orbits in $M$ are in bijective correspondence with a certain class of ideals in a poset $P$, which we call the fundamental poset.
As a set,
\begin{equation*}
  P = \{(v,k)\mid k\in \P,\; 0\leq v<k\}.
\end{equation*}
The partial order on $P$ is defined by setting
\begin{equation*}
  (v,k)\leq (v',k') \text{ if and only if } v\geq v' \text{ and } k-v\leq k'-v'.
\end{equation*}
The Hasse diagram of the fundamental poset $P$ is shown in Figure~\ref{fig:fundamental-poset}.
\begin{figure}
  \centering
  \includegraphics[width = 0.9\textwidth]{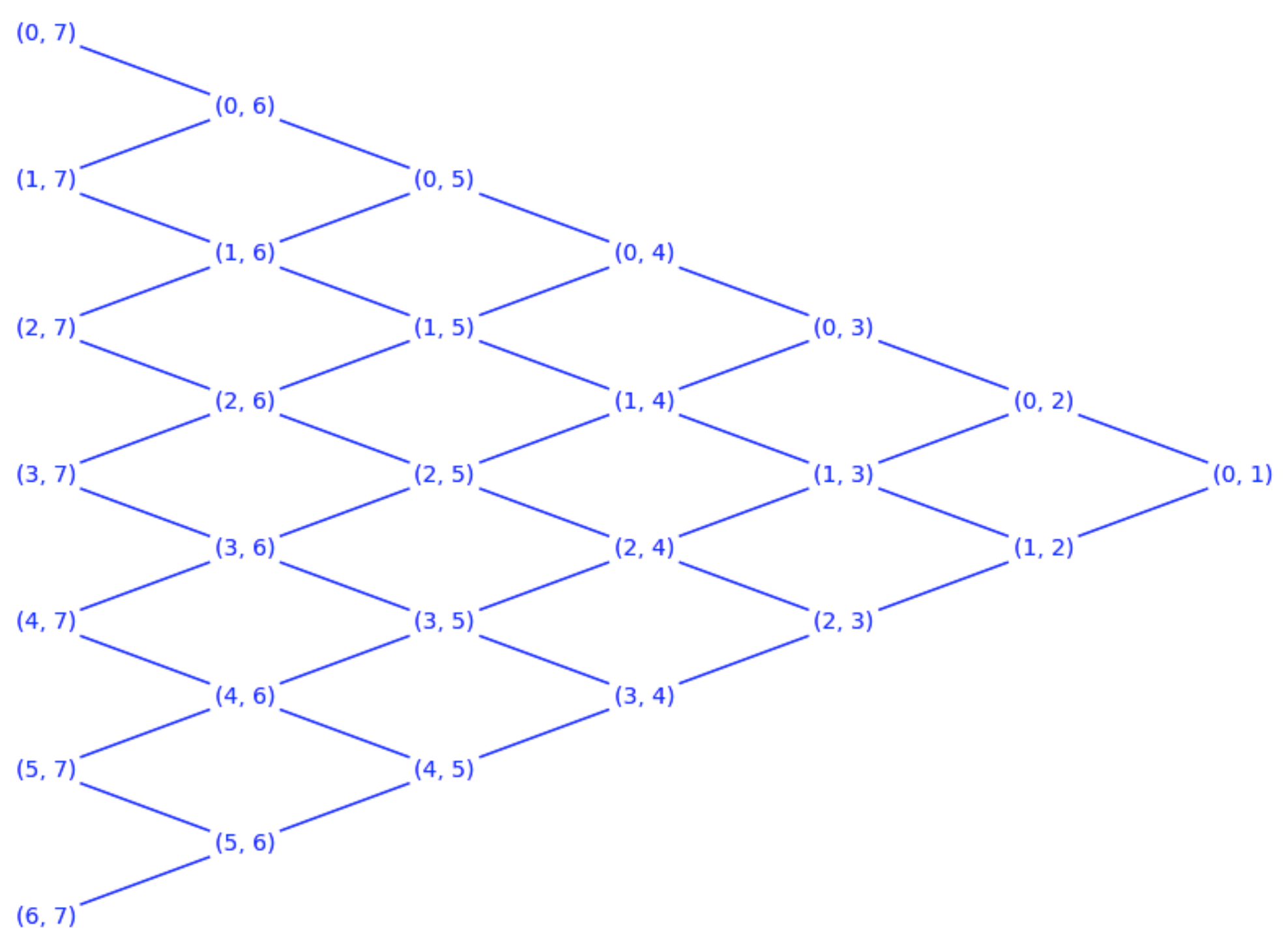}
  \caption{The Fundamental Poset}
  \label{fig:fundamental-poset}
\end{figure}
Let $J(P)$ denote the lattice of order ideals in $P$.
A typical element of $M$ from (\ref{eq:module-form}) is a vector of the form
\begin{equation}
  \label{eq:coordinates}
  m = (m_{\lambda_i,r_i}),
\end{equation}
where $i$ runs over the set $\{1,\dotsc,l\}$, and for each $i$, $r_i$ runs over a set of cardinality $m_i$.
To $m\in M$ we associate the order ideal $I(m)\in J(P)$ generated by the elements
\begin{equation*}
  (v(m_{\lambda_i,r_i}), \lambda_i)
\end{equation*}
for all pairs $(i,r_i)$ such that the coordinate $m_{\lambda_i,r_i}\neq 0$ in $R/P^{\lambda_i}$.
Here, for any $m\in M$, $v(m)$ denotes the largest $k$ for which $m\in P^kM$ (in particular, $v(0)=\infty$).

Consider for example, in the finite abelian $p$-group (or $\Z_p$-module):
\begin{equation}
  \label{eq:fapg-eg}
  M = \Z/p^5\Z\oplus\Z/p^4\Z\oplus\Z/p^4\Z\oplus\Z/p^2\Z\oplus\Z/p\Z,
\end{equation}
the order ideal $I(0,up,p^2,vp,1)$, when $u$ and $v$ are coprime to $p$, is represented inside $P$ by filled-in circles (both red and blue; the significance of the colours will be explained later) in Figure~\ref{fig:ideal-example}.
\begin{figure}
  \centering
  \includegraphics[width = 0.4\textwidth]{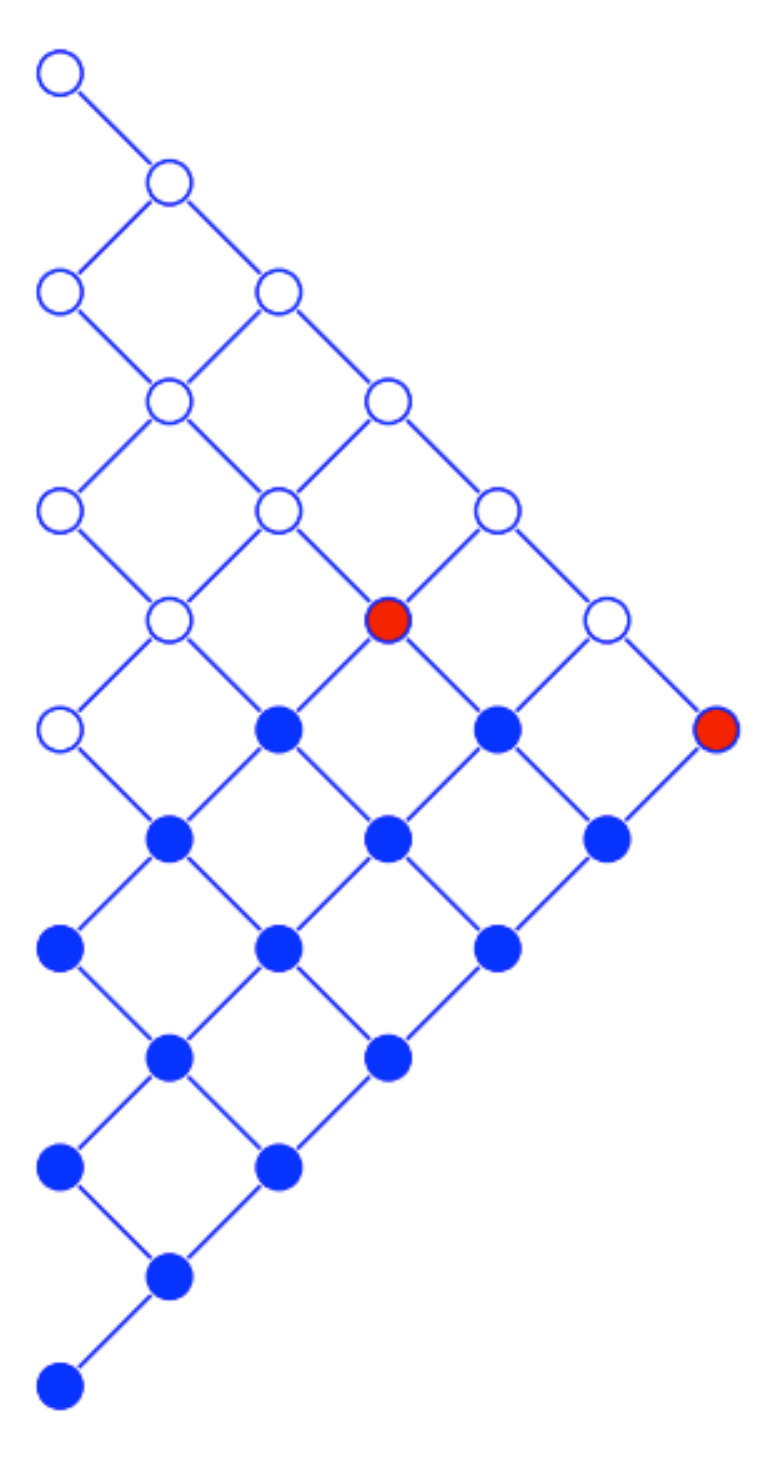}
  \caption{The ideal $I(0,up,p^2,vp,1)$ }
  \label{fig:ideal-example}
\end{figure}
Since the labels of the vertices can be inferred from their positions, they are omitted.

A key observation of \cite{MR2793603} is the following theorem:
\begin{theorem}
  \label{theorem:ideals-homomorphisms}
  Let $M$ and $N$ be two $R$-modules of bounded order.
  An element $n\in N$ is a homomorphic image of $m\in M$ (in other words, there exists a homomorphism $\phi:M\to N$ such that $\phi(m)=n$) if and only if $I(n)\subset I(m)$.
\end{theorem}
It follows that if $m'$ lies in the $G$-orbit of $m\in M$, then $I(m)=I(m')$.
It turns out that the converse is also true:
\begin{theorem}
  \label{theorem:ideals-and-orbits}
  If $I(m)=I(m')$ for any $m,m'\in M$, then $m$ and $m'$ lie in the same $G$-orbit.
\end{theorem}
Note that the orbit of $0$ corresponds to the empty ideal.
Let $J(P)_\lambda$ denote the sublattice of $J(P)$ consisting of ideals such that $\max I$ is contained in the set
\begin{equation*}
  P_\lambda = \{(v,k)\mid k=\lambda_i \text{ for some } 1\leq i \leq l\}.
\end{equation*}
Then the $G$-orbits in $M$ are in bijective correspondence with the order ideals\footnote{The lattice $J(P)_\lambda$ is isomorphic to the lattice $J(P_\lambda)$ of order ideals in the induced subposet $P_\lambda$. In \cite{MR2793603}, $J(P_\lambda)$ is used in place of $J(P)_\lambda$.} $J(P)_\lambda$.
For each order ideal $I\in J(P)_\lambda$, we use the notation
\begin{equation*}
  M^*_I = \{m\in M\mid I(m)=I\}
\end{equation*}
for the orbit corresponding to $I$.

A convenient way to think about ideals in $P$ is in terms of what we call their boundaries: for each positive integer $k$ define the boundary valuation of $I$ at $k$ to be
\begin{equation*}
  \partial_k I = \min\{v\mid (v,k)\in I\}.
\end{equation*}
We denote the sequence $\{\partial_k I\}$ of boundary valuations by $\partial I$ and call it the boundary of $I$.
This is indeed the boundary of the region with colored dots in Figure~\ref{fig:ideal-example}.

For each order ideal $I\subset P$, let $\max I$ denote its set of maximal elements.
The ideal $I$ is completely determined by $\max I$: in fact taking $I$ to $\max I$ gives a bijection from the lattice $J(P)_\lambda$ to the set of antichains in $P_\lambda$. 
For example, the maximal elements of the ideal in Figure~\ref{fig:ideal-example} are represented by red circles.

\begin{theorem}
  \label{theorem:description-of-orbits}
  The orbits $M^*_I$ consists of elements $m=(m_{\lambda_i,r_i})$ such that $v(m_{\lambda_i,r_i})\geq \partial_{\lambda_i}I$ for all $\lambda_i$ and $r_i$, and such that $v(m_{\lambda_i,r_i}) = \partial_{\lambda_i}I$ for at least one $r_i$ if $(\partial_{\lambda_i} I, \lambda_i)\in \max I$.
\end{theorem}
In other words, the elements of $M^*_I$ are those elements all of whose coordinates have valuations not less than the corresponding boundary valuation, and at least one coordinate corresponding to each maximal element of $I$ has valuation exactly equal to the corresponding boundary valuation.

In the running example with $M$ as in (\ref{eq:fapg-eg}) and $I$ as in Figure~\ref{fig:ideal-example}, the conditions for $m=(m_{5,1},m_{4,1},m_{4,2},m_{2,1},m_{1,1})$ to be in $M^*_I$ are:
\begin{itemize}
\item $v(m_{5,1})\geq 4$,
\item $\min(v(m_{4,1}),v(m_{4,2}))=1$,
\item $v(m_{2,1})\geq 1$,
\item $v(m_{1,1})=0$.
\end{itemize}
For each $I\in J(P)_\lambda$, with 
\begin{equation*}
  \max I = \{(v_1,k_1),\dotsc,(v_s,k_s)\}
\end{equation*}
define an element $m(I)$ of $M$ whose coordinates are given by
\begin{equation*}
  m_{\lambda_i,r_i} =
  \begin{cases}
    \pi^{v_j} & \text{ if } \lambda_i = k_j \text{ and } r_j = 1\\
    0 & \text{ otherwise}.
  \end{cases}
\end{equation*}
In other words, for each element $(v_j,k_j)$ of $\max I$, pick $\lambda_i$ such that $\lambda_i=k_j$.
In the summand $(R/P^{\lambda_i})^{\oplus m_i}$, set the first coordinate of $m(I)$ to $\pi^{v_j}$, and the remaining coordinates to $0$.
For example, in the finite abelian $p$-group of (\ref{eq:fapg-eg}), and the ideal $I$ of Figure~\ref{fig:ideal-example},
\begin{equation*}
  m(I) = (0,p,0,0,1).
\end{equation*}
\begin{theorem}
  \label{theorem:orbit-par}
  Let $M=M_\lambda$ be an $R$-module of bounded order as in (\ref{eq:module-form}).
  The functions $m\mapsto I(m)$ and $I\mapsto m(I)$ are mutually inverse bijections between the set of $G$-orbits in $M$ and the set of order ideals in $J(P)_\lambda$.
\end{theorem}

For any ideal $I\in J(P)$, define
\begin{equation*}
  M_I = \coprod_{\{J\in J(P)_\lambda\mid J\subset I\}} M^*_I.
\end{equation*}
This submodule, being a union of $G$-orbits, is $G$-invariant.
The description of $M_I$ in terms of valuations of coordinates and boundary valuations is very simple:
\begin{equation}
  \label{eq:M_I}
  M_I = \{ m=(m_{\lambda_i,r_i})\mid v(m_{\lambda_i,r_i})\geq \partial_{\lambda_i} I\}.
\end{equation}

Note that the map $I\mapsto M_I$ is not injective on $J(P)$.
It becomes injective when restricted to $J(P)_\lambda$.
For example, if $J$ is the order ideal in $P$ generated by $(2,6)$, $(1,4)$ and $(0,1)$, then the ideal $J$ is strictly larger than the ideal $I$ of Figure~\ref{fig:ideal-example}, but when $M$ is as in (\ref{eq:fapg-eg}), $M_I=M_J$.

The $G$-orbits in $M$ are parametrized by the finite distributive lattice $J(P)_\lambda$.
Moreover, each order ideal $I\in J(P)_\lambda$ gives rise to a $G$-invariant submodule of $M_I$ of $M$.
The lattice structure of $J(P)_\lambda$ gets reflected in the poset structure of the submodules $M_I$ when they are partially ordered by inclusion:
\begin{theorem}
  \label{theorem:characteristic-submodules}
  The map $I\mapsto M_I$ gives an isomorphism from $J(P)_\lambda$ to the poset of $G$-invariant submodules of $M$ of the form $M_I$.
\end{theorem}
In other words, for ideals $I,J\in J(P)_\lambda$,
\begin{equation*}
  M_{I\cup J} = M_I + M_J \text{ and } M_{I\cap J} = M_I\cap M_J.
\end{equation*}
In fact, when $k$ has at least three elements, every $G$-invariant submodule is of the form $M_I$, therefore $J(P)_\lambda$ is isomorphic to the lattice of $G$-invariant submodules (Kerby and Rode \cite{MR2782608}).

When $M$ is a finite $R$-module (this happens when the residue field of $R$ is finite and $\lambda\in \Lambda_0$), then the $G$-orbits in $M$ are also finite.
The cardinality of the orbit $M^*_I$ is given by (see \cite[Theorem~8.5]{MR2793603}):
\begin{equation}
  \label{eq:order-of-orbit}
  |M^*_I| = q^{[I]_\lambda}\prod_{(v_i,\lambda_i)\in \max I}(1-q^{-m_i}).
\end{equation}
Here $[I]$ denotes the number of points in $I\cap P_\lambda$ counted \emph{with multiplicity}:
\begin{equation*}
  [I]_\lambda = \sum_{i=1}^l\sum_{\{v\mid(v,\lambda_i)\in I\}} m_i.
\end{equation*}
In particular, we have:
\begin{theorem}
  \label{theorem:cardinality-of-singleton-orbit}
  Let $q$ denote the cardinality of the residue field of $R$.
  When $M$ is finite, the cardinality of $M^*_I$ is a monic polynomial in $q$ of degree $[I]_\lambda$ whose coefficients are integers which do not depend on $R$.
\end{theorem}
The formula for the cardinality of the $G$-invariant submodule is much simpler:
\begin{equation}
  \label{eq:order-of-submodule}
  |M_I| = q^{[I]_\lambda}.
\end{equation}
\section{Sum of orbits}
\label{sec:sum-of-orbits-lemma}
This section proves a combinatorial lemma on the sum of two $G$-orbits in $M$ which will be needed in Section~\ref{sec:stab-canon-forms}.
Given order ideals $I,J\subset J(P)_\lambda$, the set
\begin{equation*}
  M^*_I + M^*_J =\{ m + n \mid m\in M^*_I \text{ and } n\in M^*_J\}.
\end{equation*}
This set is clearly $G$-invariant, and therefore a union of $G$-orbits.
In this section, we determine exactly which $G$-orbits occur in $M^*_I+M^*_J$.
\begin{lemma}
  \label{lemma:sum-of-orbits}
  For $I, J\in J(P)_\lambda$, every element $(m_{\lambda_i,r_i})$ of $M^*_I+M^*_J$ satisfies the conditions
  \begin{enumerate}
  \item $v(m_{\lambda_i,r_i})\geq \max(\partial_{\lambda_i}I,\partial_{\lambda_i}J)$.
  \item If $(\partial_{\lambda_i}I,\lambda_i)\in \max I - J$, then $\min_{r_i} v(m_{\lambda_i,r_i}) = \partial_{\lambda_i}I$.
  \item If $(\partial_{\lambda_i}J,\lambda_i)\in \max J - I$, then $\min_{r_i} v(m_{\lambda_i,r_i}) = \partial_{\lambda_i}J$.
  \end{enumerate}
  If the residue field of $R$ has at least three elements, then every element of $M$ satisfying these three conditions is in $M^*_I+M^*_J$.
\end{lemma}
To see why the condition on the residue field is necessary consider the case where $M=\mathbf Z/2\mathbf Z$, and $M^*_I$ is the non-zero orbit (corresponding to the ideal $I$ in $P$ generated $(0,1)$), $M^*_I + M^*_I$ consists only of $0$.
If, on the other hand, the residue field has at least three elements, then it has non-zero elements $x$ and $y$ such that $x+y$ is also non-zero, and this phenomenon does not occur.
\begin{proof}
  [Proof of the lemma]
  Let $M_i$ denote the summand $(R/P^{\lambda_i})^{\oplus m_i}$ of $M$ in the decomposition (\ref{eq:module-form}). 
  Let $M_i^* = M_i - \pi M_i$.
  By Theorem~\ref{theorem:description-of-orbits} it suffices to show that
  \begin{equation}
    \label{eq:sum-of-orbits}
    \pi^k M_i^* + \pi^l M_i^* = 
    \begin{cases}
      \pi^{\min(k,l)}M_i^* & \text{ if } k\neq l,\\
      \pi^kM_i & \text{ if $k=l$ and $|R/P|\geq 3$.}
    \end{cases}
  \end{equation}
  which follows from the well-known non-Archimedean inequality
  \begin{equation*}
    v(x+y) \geq \min (v(x),v(y)),
  \end{equation*}
  and the fact that strict inequality is possible only if $v(x)=v(y)$.
\end{proof}
Together with Theorem~\ref{theorem:description-of-orbits}, the above lemma gives the following description of the set of orbits which occur in $M^*_I+M^*_J$:
\begin{theorem}
  \label{theorem:sum-of-orbits}
  Assume that the residue field of $R$ has at least three elements.
  For ideals $I,J\in J(P)_\lambda$,
  \begin{equation*}
    M^*_I + M^*_J = \coprod_{K\subset I\cup J,\;\max K \supset (\max I - J)\cup (\max J - I)} M^*_K.
  \end{equation*}
\end{theorem}
In the following lemma the restriction on the residue field of $R$ in Lemma~\ref{lemma:sum-of-orbits} is not needed:
\begin{lemma}
  For ideals $I$ and $J$ in $J(P)_\lambda$, an element $(m_{\lambda_i,r_i})$ is in $M^*_I + M_J$ if and only if the following conditions are satisfied:
  \begin{enumerate}
  \item $v(m_{\lambda_i,r_i})\geq \min(\partial_{\lambda_i}I,\partial_{\lambda_i}J)$.
  \item If $(\partial_{\lambda_i}I,\lambda_i)\in \max I - J$, then $\min_{r_i} v(m_{\lambda_i,r_i}) = \partial_{\lambda_i}I$.
  \end{enumerate}
\end{lemma}
\begin{proof}
  The proof is similar to that of Lemma~\ref{lemma:sum-of-orbits}, except that instead of (\ref{eq:sum-of-orbits}), we use:
  \begin{equation*}
    \pi^k M_i + \pi^l M_i^* = 
    \begin{cases}
      \pi^k M_i & \text{ if $k\leq l$},\\
      \pi^lM_i^* & \text{ if $k>l$.}
    \end{cases}
  \end{equation*}
\end{proof}
The above lemma allows us to describe the sum of an orbit and a characteristic submodule:
\begin{theorem}
  \label{theorem:sum-of-orbit-and-submodule}
  For ideals $I,J\in J(P)_\lambda$,
  \begin{equation}
    \label{eq:sum-of-orbit-and-submodule}
    M_I^*+M_J = \coprod_{K\subset I\cup J,\;\max K \supset \max I - J}M^*_K.
  \end{equation}
\end{theorem}
\section{Stabilizers of $m(I)$'s}
\label{sec:stab-canon-forms}

By Theorem~\ref{theorem:orbit-par}, every $G$-orbit of pairs of elements $(m_1,m_2)\in M^2$ contains a pair of the form $(m(I),m)$, for some $I\in J(P)_\lambda$.
Now fix an ideal $I\in J(P)_\lambda$.
Let $G_I$ denote the stabilizer in $G$ of $m(I)$.
Then the $G$-orbits of pairs in $M^2$ which contain an element of the form $(m(I),m)$ are in bijective correspondence with $G_I$-orbits in $M$.
In this section, we give a description of $G_I$ which facilitates the classification of $G_I$-orbits in $M$.

The main idea here is to decompose $M$ into a direct sum of two $R$-modules (this decomposition depends on $I$):
\begin{equation}
  \label{eq:decomposition}
  M = M'\oplus M'',
\end{equation}
where $M'$ consists of those cyclic summands in the decomposition (\ref{eq:module-form}) of $M$ where $m(I)$ has non-zero coordinates, and $M''$ consists of the remaining cyclic summands.
In the running example with $M$ given by (\ref{eq:fapg-eg}), and $I$ the ideal in Figure~\ref{fig:ideal-example}, we have
\begin{equation*}
  M' = \Z/p^4\Z\oplus \Z/p\Z,\quad M'' = Z/p^5\Z\oplus\Z/p^4\Z\oplus\Z/p^2\Z.
\end{equation*}
Note that $m(I)\in M'$.
The reason for introducing this decomposition is that the description of the stabilizer of $m(I)$ in the automorphism group of $M'$ is quite nice:
\begin{lemma}
  The stabilizer of $m(I)$ in $\Aut_R(M')$ is
  \begin{equation*}
    G'_I = \{\id{M'}+u\mid u\in \End_RM'\text{ satisfies }u(m(I))=0\}.
  \end{equation*}
\end{lemma}
\begin{proof}
  Obviously, the elements of $G'_{I}$ are all the elements of $\End_RM'$ which map $m(I)$ to itself.
  The only thing to check is that they are all invertible.
  For this, it suffices to show that if $u(m(I))=0$, then $u$ is nilpotent, which will follow from Lemma~\ref{lemma:nilpotent} below.
\end{proof}
\begin{lemma}
  \label{lemma:nilpotent}
  For any $R$-module of the form 
  \begin{equation*}
    L = R/P^{\mu_1}\oplus\dotsb\oplus R/P^{\mu_m},
  \end{equation*}
  with $\mu_1>\dotsb>\mu_m$,
  and $x=(\pi^{v_1},\dotsc,\pi^{v_m})\in L$ such that the set
  \begin{equation*}
    (v_1,\mu_1),\dotsc,(v_m,\mu_m)
  \end{equation*}
  is an antichain in $P$, if $u\in \End_RL$ is such that $u(x)=0$, then $u$ is nilpotent. 
\end{lemma}
\begin{proof}
  Write $u$ as a matrix $(u_{ij})$, where $u_{ij}:R/P^{\lambda_j}\to R/P^{\lambda_i}$.
  We have
  \begin{equation*}
    u(\pi^{v_1},\dotsc,\pi^{v_m})_i = u_{ii}(\pi^{v_i}) + \sum_{j\neq i}u_{ij}(\pi^{v_j}) =0,
  \end{equation*}
  for $1\leq i\leq m$.
  If $u_{ii}1$ is a unit, then $u_{ii}\pi^{v_i}$ has valuation $v_i$, hence at least one of the summands $u_{ij}\pi^{v_j}$ must have valuation $v_i$ or less.
  It follows from Theorem~\ref{theorem:ideals-homomorphisms} (applied to $M=R/P^{\mu_j}$ and $N=R/P^{\mu_i}$) that $(v_i,\mu_i)\leq (v_j,\mu_j)$ contradicting the antichain hypothesis.
  Thus, for each $i$, $u_{ii}(1)\in PR/P^{\mu_i}$.
  It follows that $u$ lies in the radical of the ring $\End_R L$ (see Dubey, Prasad and Singla \cite[Section~6]{DPS}), and therefore $u$ is nilpotent.
\end{proof}
Every endomorphism of $M$ can be written as a matrix $\mat xyzw$, where $x:M'\to M'$, $y:M''\to M'$, $z:M''\to M'$ and $w:M''\to M''$ are homomorphisms.

We are now ready to describe the stabilizer of $m(I)$ in $M$:
\begin{theorem}
  \label{theorem:stabilizer}
  The stabilizer of $m(I)$ in $G$ consists of matrices of the form
  \begin{equation*}
    \mat{\id{M'}+u}yzw, 
  \end{equation*}
  where $u\in \End_RM'$ satisfies $u(m(I))=0$,  $y\in \Hom_R(M'',M')$ is arbitrary, $z\in \Hom_R(M',M'')$ satisfies $z(m(I))=0$, and $w\in \End_R(M'')$ is invertible.
\end{theorem}
\begin{proof}
  Clearly, all the endomorphisms of $M$ which fix $m(I)$ are of the form stated in the theorem, except that $w$ need not be invertible.
  We need to show that the invertibility of such an endomorphism is equivalent to the invertibility of $w$.

  To begin with, consider the case where $M=(R/P^k)^n$ for some positive integer $k$ and some cardinal $n$.
  Then, if $m(I)\neq 0$ (the case $m(I)=0$ is trivial), then $M'=R/P^k$, and $M''=(R/P^k)^{n-1}$.
  The endomorphisms which fix $m(I)$ are all of the form
  \begin{equation*}
    \mat{1+u}yzw,
  \end{equation*}
  where $u$, and each coordinate of $z$ lies in $P$.
  Such endomorphisms, being block upper-triangular modulo $P$, are invertible if and only if $w$ is invertible, proving the claim when $M=(R/P^k)^n$.
  In general, $M$ is a sum of such modules, and an endomorphism of $M$ is invertible if and only if its diagonal block corresponding to each of these summands is invertible.
  Therefore the claim follows in general as well.
\end{proof}
\begin{cor}
  [Independence of multiplicities larger than two]
  \label{cor:independence of multiplicities}
  Consider the partition $\lambda^{(m)}$ derived from $\lambda$ by:
  \begin{equation*}
    \lambda^{(m)} = (\lambda_1^{\min(m_1,m)}, \lambda_2^{\min(m_2,m)}, \dotsc, \lambda_l^{\min(m_l,m)}).
  \end{equation*}
  Let $M_m$ denote the $R$-module corresponding to $\lambda^{(m)}$, with automorphism group $G_m$.
  Then the standard inclusion map $M_2\hookrightarrow M$ induces a bijection
  \begin{equation}
    \label{eq:inclusion}
    G_2\backslash M_2\times M_2 \tilde\to G\backslash M\times M. 
  \end{equation}
\end{cor}
\begin{proof}
  We shall use the fact that the canonical forms $m(I)$ of Theorem~\ref{theorem:orbit-par} lie in $M_1\subset M$.
  Thus given a pair $(x,y)\in M\times M$, we can reduce $x$ to $m(I)\in M_1$ using automorphisms of $M$.
  Theorem~\ref{theorem:stabilizer} shows that, while preserving $m(I)$, automorphisms of $M$ can be used to further reduce $y$ to an element of $M'\oplus M''_1\subset M_2$.
  This proves the surjectivity of the map in (\ref{eq:inclusion}).

  To see injectivity, suppose that two pairs $(x_1,y_1)$ and $(x_2,y_2)$ in $M_2\times M_2$ lie in the same $G$-orbit.
  Since $M_2$ is a direct summand of $M$, we can write $M=M_2\oplus N$.
  If $g\in G$ has matrix $\mat {g_{11}}{g_{12}}{g_{21}}{g_{22}}$ with respect to this decomposition, then $g_{11}\in G_2$ also maps $(x_1,y_1)\in M_2\times M_2$ to $(x_2,y_2)\in M_2\times M_2$. 
\end{proof}
\begin{remark}
  Corollary~\ref{cor:independence of multiplicities} and its proof remain valid if we restrict ourselves to $G$-orbits in $M^*_I\times M^*_J$ for order ideals $I,J\in J(P)_\lambda$.
\end{remark}
\section{The stabilizer orbit of an element}
\label{sec:stabilizer-orbit-an}
Let $G_I$ denote the stabilizer of $m(I)\in M$.
Write each element $m\in M$ as $m=(m',m'')$ with respect to the decomposition (\ref{eq:decomposition}) of $M$.
Also, for any $m'\in M'$, let $\bar m'$ denote the image of $m'$ in $M'/Rm(I)$.

Theorem~\ref{theorem:stabilizer} allows us to describe the orbit of $m$ under the action of $G_I$, which is the same as describing the $G$-orbits in $M^2$ whose first component lies in the orbit $M^*_I$ of $m(I)$.
\begin{theorem}
  \label{theorem:sub-orbit}
  Given $m$ and $n$ in $M$, $n$ lies in the $G_I$-orbit of $m$ in $M$ if and only if the following conditions hold:
  \begin{enumerate}
  \item $n'\in m' + M'_{I(\bar m')\cup I(m'')}$.
  \item \label{item:second-condition} $n''\in {M''}^*_{I(m'')}+M''_{I(\bar m')}$.
  \end{enumerate}
\end{theorem}
\begin{proof}
  By Theorem~\ref{theorem:stabilizer}, $n$ lies in the $G_I$-orbit of $m$ if and only if
  \begin{equation*}
    n' = m' + \bar u(\bar m') + y(m'') \text{ and } n'' = \bar z(\bar m') + w(m'')
  \end{equation*}
  for homomorphisms $\bar u \in \Hom_R(M'/Rm(I),M')$, $y \in \Hom_R(M'',M')$,\linebreak $z\in \Hom_R(M'/Rm(I),M'')$ and $w\in \Aut_R(M'')$.
  By Theorems~\ref{theorem:ideals-homomorphisms} and~\ref{theorem:ideals-and-orbits}, this means
  \begin{equation*}
    n' \in  m' + M'_{I(\bar m')} + M'_{I(m'')} \text{ and } n''\in M''_{I(\bar m')} + {M''}^*_{I(m'')}.
  \end{equation*}
  By the remark following Theorem~\ref{theorem:characteristic-submodules}, $M'_{I(\bar m')}+M'_{I(m'')} = M'_{I(\bar m')\cup I(m'')}$, giving the conditions in the lemma.
\end{proof}
Given $m=(m',m'')\in M$, the ideals $I(\bar m')$ and $I(m'')$ may be regarded as combinatorial invariants of $m$.
Now suppose that the residue field $k$ of $R$ is finite of order $q$.
We can now show that, having fixed these combinatorial invariants, the cardinality of the orbit of $m$ is a polynomial in $q$ whose coefficients are integers which do not depend on $R$.
Also, the number of elements of $M$ having these combinatorial invariants is a polynomial in $q$ whose coefficients are integers which do not depend on $R$.
Using these observations, we will be able to conclude that the number of orbits of pairs in $M$ is a polynomial in $q$ whose coefficients are integers which do not depend on $R$.

Let $\lambda'/I$ denote the partition corresponding to the isomorphism class of $M'/Rm(I)$.
The partition $\lambda'/I$ is completely determined by the partition $\lambda'$ and the ideal $I\in J(P)_{\lambda'}$, and is independent of $R$ (see Lemma~\ref{lemma:quotient}).
\begin{theorem}
  \label{theorem:sub-orbit-order}
  Fix $J\in J(P)_{\lambda'/I}$, $K\in J(P_\lambda'')$.
  Then the cardinality of the $G_I$-orbit of any element $m=(m',m'')$ such that $I(\bar m') = J$ and $I(m'')=K$ is given by
  \begin{equation}
    \label{eq:cardinality-of-orbit}
    \alpha_{I,J,K} = |M'_{J\cup K}|\Bigg(\sum_{K'\subset JUK,\;\max K'\supset \max K - J} |{M''}^*_{K'}|\Bigg).
  \end{equation}
\end{theorem}
\begin{proof}
  This is a direct consequence of Theorems~\ref{theorem:sum-of-orbit-and-submodule} and~\ref{theorem:sub-orbit}.
\end{proof}
Applying Theorem~\ref{theorem:cardinality-of-singleton-orbit} and (\ref{eq:order-of-submodule}) to Theorem~\ref{theorem:sub-orbit-order} gives:
\begin{theorem}
  \label{theorem:cardinality-of-orbit}
  When $M$ is finite (and $q$ denotes the cardinality of the residue field of $R$), the cardinality of every $G_I$-orbit in $M$ is of the form $\alpha_{I,J,K}=\alpha_{I,J,K}(q)$ for some $J\in J(P)_{\lambda'/I}$ and some $K\in J(P)_{\lambda''}$.
Each $\alpha_{I,J,K}(q)$ is a monic polynomial in $q$ of degree $[J\cup K]_\lambda$ whose coefficients are integers which are independent of the ring $R$.
\end{theorem}
If the sets
\begin{equation*}
  X_{I,J,K}=\{(m',m'')\in M\mid I(\bar m')=J\text{ and } I(m'')=K\}
\end{equation*}
were $G_I$-stable, we could have concluded that $X_{I,J,K}$ consists of 
\begin{equation*}
  \frac{|X_{I,J,K}|}{\alpha_{I,J,K}}
\end{equation*}
many orbits, each of cardinality $\alpha_{I,J,K}$.
However, $X_{I,J,K}$ is not, in general, $G_I$-stable (this can be seen by viewing the condition \ref{item:second-condition} in the context of Theorem~\ref{theorem:sum-of-orbit-and-submodule}).
The following lemma gives us a way to work around this problem:
\begin{lemma}
  \label{lemma:counting-parts}
  Let $S$ be a finite set with a partition $S= \coprod_{i=1}^N S_i$ (for the application we have in mind, these will be the $G_I$-orbits in $M$).
  Suppose that $S$ has another partition $S=\coprod_{j=1}^{Q} T_j$, such that there exist positive integers $n_1,n_2,\dotsc,n_Q$  for which, if $x\in T_j\cap S_i$, then $|S_i|=n_j$ (in our case, the $T_j$'s will be the sets $X_{I,J,K}$).
  Then the number of $i\in \{1,\dotsc, N\}$ such that $|S_i|=n$ is given by
  \begin{equation*}
    \frac 1n\sum_{\{j\mid n_j = n\}}|T_j|.
  \end{equation*}
\end{lemma}
\begin{proof}
  Note that
  \begin{equation*}
    \coprod_{\{j\mid n_j=n\}}|T_j|
  \end{equation*}
  is the union of all the $S_i$'s for which $|S_i|=n$.
\end{proof}
Taking $S$ to be the set $M$, the $S_i$'s to be the $G_I$-orbits in $M$, and $T_j$'s to be the sets $X_{I,J,K}$ in Lemma~\ref{lemma:counting-parts} gives:
\begin{theorem}
  \label{theorem:G_I-orbits-cardinalities}
  Let $\alpha(q)$ be a monic polynomial in $q$ with integer coefficients.
  Then the number of $G_I$-orbits in $M$ with cardinality $\alpha(q)$ is
  \begin{equation*}
    N_\alpha(q) = \frac 1{\alpha(q)} \sum_{\{(I,J,K)\mid \alpha_{I,J,K}(q)=\alpha(q)\}} |X_{I,J,K}|.
  \end{equation*}
\end{theorem}
Since $\alpha(q)$ and $|X_{I,J,K}|$ are polynomials in $q$, the number $N_{\alpha}(q)$ of $G_I$-orbits in $M$ of cardinality $\alpha(q)$ is a rational function in $q$.
The following lemma will show that it is in fact a polynomial in $q$ with integer coefficients:
\begin{lemma}
  \label{lemma:polynomiality}
  Let $r(q)$ and $s(q)$ be polynomials in $q$ with integer coefficients.
  Suppose that $r(q)/s(q)$ takes integer values for infinitely many values of $q$.
  Then $r(q)/s(q)$ is a polynomial in $q$ with rational coefficients.
  If, in addition $s(q)$ is monic, then $r(q)$ has integer coefficients.
\end{lemma}
The proof, being fairly straightforward, is omitted.
\begin{example}
  \label{example:max-ideal}
  Consider an arbitrary $\lambda\in \Lambda$, and take $I$ to be the maximal ideal in $J(P)_\lambda$ (this is the ideal in $P$ generated by $P_\lambda$).
  Then, in the notation of (\ref{eq:lambda}),
  \begin{equation*}
    \lambda' = (\lambda_1), \quad \lambda'' = (\lambda_1^{m_1-1},\lambda_2^{m_2},\dotsc,\lambda_l^{m_l}).
  \end{equation*}
  The element $m(I)$ is a generator of $M'$, and so $M'/Rm(I)=0$.
  It follows that the only possibility for the ideal $J\in J(P)_{\lambda'/I}$ is $J=\emptyset$.
    As a result, the only combinatorial invariant of a $G_I$-orbits in $M$ is $K\in J(P)_{\lambda''}$.
    We have 
    \begin{equation*}
      \alpha_{I,\emptyset,K}(q) = |M'_K||{M''}^*_K|.
    \end{equation*}
    On the other hand,
    \begin{equation*}
      |X_{I,\emptyset,K}| = q^{\lambda_1}|{M''}^*_K|.
    \end{equation*}
    Therefore, given a polynomial $\alpha(q)$, the number of $G_I$-orbits of cardinality $\alpha(q)$ is
    \begin{equation*}
      \sum_{\{K\in J(P)_{\lambda''}\mid \alpha_{I,\emptyset,K}=\alpha(q)\}} \frac{q^{\lambda_1}}{|M'_K|}.
    \end{equation*}
    Since $K=\emptyset$ is the only ideal in $J(P)_{\lambda''}$ for which $|M'_K|=1$, it turns out that the total number of $G_I$-orbits in $M_I\times M$ is a monic polynomial in $q$ of degree $\lambda_1$.

    For example, if $\lambda = (2,1^{m_2})$, then the number of $G_I$-orbits in $M$ is $q^2+q$, and if $\lambda=(2^{m_1},1^{m_2})$ with $m_1>1$, then the number of $G_I$-orbits in $\lambda$ is $q^2+2q+1$.
\end{example}
\begin{example}
  \begin{table}
    \label{tab:cards-numbers}
    \begin{equation*}
      \begin{array}{|c|c|}
        \hline
        \text{Cardinality} & \text{Number of Orbits}\\
        \hline
        1 & q^{3} \\
        (q - 1)  q^{7} & (q - 1)  q \\
        (q - 1)  q^{12} & (q - 1) \\
        q^{4} & (q - 1)  q^{2} \\
        (q - 1)^{2}  q^{11} & 1 \\
        (q - 1)^{2}  q^{8} & q \\
        (q - 1)^{2}  q^{10} & 1 \\
        (q - 1)  q^{2} & q^{2} \\
        (q - 1)^{2}  q^{6} & q \\
        (q - 1)^{2}  q^{3} & q^{2} \\
        (q - 1)^{2}  q^{5} & q \\
        (q - 1) & q^{3} \\
        (q - 1)  q^{15} & 1 \\
        (q - 1)  q^{5} & q \\
        q^{9} & (q - 1)  q \\
        (q - 1)  q^{8} & q \\
        (q - 1)  q^{14} & 1 \\
        (q - 1)  q^{11} & (q - 1) \\
        (q - 1)  q^{6} & q^{2} \\
        (q - 1)  q^{4} & (q - 1)  q^{2} \\
        (q - 1)  q^{3} & 2q^{2} \\
        (q - 1)  q^{9} & q^{2} \\
        (q - 1)  q^{10} & q \\
        \hline
      \end{array}
    \end{equation*}
    \caption{Cardinalities and numbers of $G_I$-orbits}
  \end{table}
  Now consider the case where $\lambda=(5,4,4,2,1)$ and $I$ is the ideal of Figure~\ref{fig:ideal-example}.
  Then the first column of Table~\ref{tab:cards-numbers} gives all the possible cardinalities for $G_I$-orbits in $M$.
  The corresponding entry of the second column is the number of orbits with that cardinality.
  The total number of $G_I$-orbits in $M$ is given by the polynomial
  \begin{equation*}
    4q^3 + 6q^2 + 6q + 2.
  \end{equation*}
  This data was generated using a computer program written in sage. In general the total number of $G_I$ orbits in $M$ need not be a polynomial with positive integer coefficients, for example, take $\lambda = (2)$ (so $M=R/P^2R$) and $I$ is the ideal generated by $(1,2)$ (the corresponding orbit in $M$ contains $\pi$).
\end{example}
The above results can be summarized to give the following Theorem:
\begin{theorem}
  \label{theorem:main1}
  Let $R$ be a discrete valuation ring with finite residue field.
  Fix $\lambda \in \Lambda_0$ and take $M$ as in (\ref{eq:module-form}).
  Let $G$ denote the group of $R$-module automorphisms of $M$.
  Fix an order ideal $I\in J(P)_\lambda$ (and hence the $G$-orbit $M^*_I$ in $M$).
  \begin{enumerate}
  \item The cardinality of each $G$-orbit in $M^*_I\times M$ is a monic polynomial in $q$ whose coefficients are integers.
  \item Given a monic polynomial $\beta(q)$ with integer coefficients, the number of $G$-orbits in $M^*_I\times M$ of cardinality $\beta(q)$ is a polynomial in $q$ with coefficients that are integers which do not depend on $R$.
  \end{enumerate}
\end{theorem}
For the total number of orbits in $M\times M$, we have:
\begin{theorem}
  \label{theorem:degree}
  Let $R$ be a discrete valuation ring with finite residue field of order $q$.
  Fix $\lambda \in \Lambda$ and take $M$ as in (\ref{eq:module-form}).
  Let $G$ denote the group of $R$-module automorphisms of $M$.
  Then there exists a monic polynomial $n_\lambda(q)$ of degree $\lambda_1$ with integer coefficients (which do not depend on $R$ or $q$) such that the number of $G$-orbits in $M\times M$ is $n_\lambda(q)$.
\end{theorem}
\begin{proof}
  The only thing that remains to be proved is the assertion about the degree of $n_\lambda(q)$.
  By Theorem~\ref{theorem:G_I-orbits-cardinalities}, 
  \begin{equation*}
    \deg n_\lambda (q) = \max_{I,J,K} (\deg|X_{I,J,K}|-\deg\alpha_{I,J,K}(q)).
  \end{equation*}
  Recalling the definitions of $X_{I,J,K}$ and $\alpha_{I,J,K}(q)$, we find that we need to show that
  \begin{equation*}
    [J\cup K]_{\lambda'/I} + \log_q|Rm(I)| + [K]_{\lambda''} \leq \lambda_1 + [J\cup K]_\lambda.
  \end{equation*}
  Observe that $[J\cup K]_\lambda = [J\cup K]_{\lambda'}+[J\cup K]_{\lambda''}$, and $[K]_{\lambda''}\leq [J\cup K]_{\lambda''} $.
  Moreover, it turns out that $[J\cup K]_{\lambda'/I} = [J\cup K]_{\lambda'}$ (see Lemma~\ref{lemma:lovely} below).
  Therefore, the inequality to be proved reduces to $\log_q|Rm(I)|\leq \lambda_1$, which is obviously true.
  Furthermore, if equality holds, then $|Rm(I)|=\lambda_1$, which is only possible if $I$ is the maximal ideal in $J(P)_\lambda$, which was considered in Example~\ref{example:max-ideal}, where a monic polynomial of degree $0$ was obtained.
\end{proof}
\begin{lemma}
  \label{lemma:lovely}
  For any ideal $J\in J(P)_{\lambda'/I}$, 
  \begin{equation*}
    [J]_{\lambda'/I} = [J]_{\lambda'}.
  \end{equation*}
\end{lemma}
\begin{proof}
  The partition $\lambda'/I$ is described in Lemma~\ref{lemma:quotient}.
  Observe that
  \begin{equation*}
    k_1\geq v_1+k_2-v_2 \geq k_2 \geq v_2+k_3-v_2 \geq \dotsb \geq v_{s-1}+k_s-v_s \geq v_s,
  \end{equation*}
  In other words, the parts of $\lambda'/I$ alternate with the parts of $\lambda'$.
  For each ideal $J\in J(P)_{\lambda'/I}$, the contribution of $J$ to $[J]_{\lambda'/I}$ in a given chain $(*,v_i+k_{i+1}-v_{i+1})\subset P_{\lambda'/I}$  (or $(*,v_s)\subset P_{\lambda'/I}$) is equal to its contribution to $[J]_{\lambda'}$ in the chain $(*,k_i)\subset P_{\lambda'}$ (resp. $(*,k_s)\subset P_{\lambda'}$).
It follows that $[J]_{\lambda'} = [J]_{\lambda'/I}$.
\end{proof}
\section{Orbits in $M^*_I\times M^*_L$}
\label{sec:module-mrmi}
In order to refine Theorem~\ref{theorem:main1} to the enumeration of $G$-orbits in $M^*_I\times M^*_L$ for a pair of order ideals $(I,L)\in J(P)_\lambda^2$, we need to repeat the calculations in Section~\ref{sec:stabilizer-orbit-an} with $X_{I,J,K}$ replaced by its subset
\begin{equation*}
  X_{I,J,K,L} = \{m\in X_{I,J,K}\mid m\in M^*_L\}.
\end{equation*}
Thus our goal is to show that $|X_{I,J,K,L}|$ is a polynomial in $q$ whose coefficients are integers which do not depend on $R$.
By using M\"obius inversion on the lattice $J(P)_{\lambda}$, it suffices to show that 
\begin{equation*}
  Y_{I,J,K,L} = \{m\in X_{I,J,K}\mid m\in M_L\}
\end{equation*}
has cardinality polynomial in $q$ whose coefficients are integers which do not depend on $R$.
This is easier, because $m=(m',m'')\in M_L$ if and only if $m'\in M_L'$ and $m''\in M_L''$.
If $(m',m'')\in _{I,J,K,L}$, we also have that $m''\in {M''_K}^*$.
Thus $Y_{I,J,K,L}=\emptyset$ unless $K\subset L$, in which case
\begin{equation*}
  |Y_{I,J,K,L}| = \#\{m'\in M'_L\mid I(\bar m')=J\}|{M''_K}^*|.
\end{equation*}
Therefore, we are reduced to proving the following lemma:
\begin{lemma}
  \label{lemma:image-orbit}
  The cardinality of the set
  \begin{equation*}
    \{m'\in M'\mid m'\in M'_L \text{ and } I(\bar m')=J\}
  \end{equation*}
  is a polynomial in $q$ whose coefficients are integers which do not depend on $R$.
\end{lemma}
\begin{proof}
  Let $\bar M'$ denote the quotient $M'/Rm(I)$ (so $\bar M'$ is isomorphic to $M_{\lambda'/I}$ in the notation of Section~\ref{sec:stabilizer-orbit-an}).
  Suppose that $\max I = \{(v_1,k_1),\dotsc,(v_s,k_s)\}$.
  Then 
  \begin{equation*}
    M' = R/P^{k_1}\oplus \dotsb \oplus R/P^{k_s}.
  \end{equation*}
  \begin{lemma}
    \label{lemma:quotient}
    Let $\lambda'/I$ denote the partition given by
    \begin{equation*}
      \lambda'/I = (v_1+k_2-v_2, v_2+k_3-v_3,\dotsc, v_{s-1}+k_s-v_s,v_s).
    \end{equation*}
    and $M_{\lambda'/I}$ be the corresponding $R$-module as given by (\ref{eq:module-form}).
    If $Q$ is the matrix
    \begin{equation*}
      Q= 
      \begin{pmatrix}
        1 & -\pi^{v_1-v_2} & 0 & \cdots & 0 &0\\
        0 & 1 & -\pi^{v_2-v_3} & \cdots & 0 &0\\
        0 & 0 & 1 & \cdots & 0 & 0\\
        \vdots & \vdots & \vdots & \ddots & \vdots & \vdots\\
        0 & 0 & 0 & \cdots & 1 & -\pi^{v_{s-1}-v_s}\\
        0 & 0 & 0 & \cdots & 0 & 1
      \end{pmatrix}
    \end{equation*}
    then the isomorphism $R^s\to R^s$ whose matrix is $Q$ descends to a homomorphism $\bar Q:M'\to M_{\lambda'/I}$ such that $\ker \bar Q\supset Rm(I)$.
    The induced homomorphism $M'/Rm(I)\to M_{\lambda'/I}$ is an isomorphism of $R$-modules.
  \end{lemma}
  \begin{proof}
    Let $e_1,\dotsc,e_s$ denote the generators of $M'$, and $f_1,\dotsc,f_s$ denote the generators of $M_{\lambda'/I}$.
    Then
    \begin{equation*}
      Q\tilde e_j = 
      \begin{cases}
        \tilde f_1 & \text{ for } j=1,\\
        -\pi^{v_{j-1}-v_j}\tilde f_{j-1} + \tilde f_j \text{ for } 1<j\leq n.
      \end{cases}
    \end{equation*}
    Here $\tilde e_j$ (or $\tilde f_i$) denotes the standard lift of $e_i$ (or $f_j$) to $R^s$.
    By using the inequalities $k_j>v_j+k_{j+1}-v_{j+1}$ for $1\leq j<s$ and $k_s\geq v_s$, one easily verifies that $Q(\pi^{k_j}\tilde e_j)$ is $0$ in $M_{\lambda'/I}$.
    Therefore $Q$ induces a well-defined $R$-module homomorphism $\bar Q:M'\to M_{\lambda'/I}$.
    Now 
    \begin{eqnarray*}
      \bar Q(m(I)) & = & \bar Q(\sum \pi^{v_j}e_j)\\
      & = & \pi^{v_1}f_1 + (-\pi^{v_2+v_1-v_2}f_1 + \pi^{v_2}f_2) + (-\pi^{v_3+v_2-v_3}f_2 + \pi^{v_3}f_2) +\\
      & & \dotsb + (-\pi^{v_s+v_{s-1}-v_s}f_{s-1} + \pi^{v_s}f_s)\\
      & = & 0.
    \end{eqnarray*}
    Therefore $\bar Q$ induces a homomorphism $M'/Rm(I)\to M_{\lambda'/I}$.
    Because $Q\in SL_s(R)$, $\bar Q$ is onto.
    When the residue field of $R$ is finite, one easily verifies that $|Rm(I)||M_{\lambda'/I}| = |M'|$, whereby $\bar Q$ is an isomorphism.
    Indeed, $|Rm(I)| = q^{k_1-v_1}$, $|M'|=q^{|\lambda'|}$ and $|M_{\lambda'/I}| = q^{|\lambda'/I|} = q^{v_1+k_2+\dotsb+k_s}$.
    In general, this argument using cardinalities can be easily replaced by an argument using the lengths of modules of $R$.
  \end{proof}
  We now return to the proof of Lemma~\ref{lemma:image-orbit}.
  Using M\"obius inversion on the lattice $J(P_{\lambda'/I})$, in order to prove Lemma~\ref{lemma:image-orbit}, it suffices to show that the cardinality of the set
  \begin{equation*}
    S=\{m'\in M'\mid m'\in M'_L \text{ and } \bar m' \in (M_{\lambda/I})_J\}
  \end{equation*}
  is a polynomial in $q$ whose coefficients are integers which do not depend on $R$.
  Write $m'\in M'$ as $m'_1e_1+\dotsb+m'_se_s$, and $n\in M_{\lambda/I}$ as $n_1f_1+\dotsb+n_sf_s$.
  By (\ref{eq:M_I}) and Lemma~\ref{lemma:quotient} $S$ consists of elements $m'\in M'$ such that
  \begin{align*}
    v(m'_i) & \geq \partial_{k_i}L & \text{ for } i=1,\dots,s,\\
    v(\bar Q(m')_i) & \geq \partial_{v_i+k_{i+1}-v_{i+1}}J & \text{ for } i=1,\dotsc,s-1, \text{ and}\\
    v(\bar Q(m')_s)& \geq \partial_{v_s}J, &
  \end{align*}
  which can be rewritten as
  \begin{align*}
    v(m'_i) & \geq \partial_{k_i}L & \text{ for } i=1,\dots,s,\\
    v(m'_i-\pi^{v_i-v_{i+1}}m'_{i+1}) & \geq \partial_{v_i+k_{i+1}-v_{i+1}}J & \text{ for } i=1,\dotsc,s-1, \text{ and}\\
    v(m'_s)& \geq \partial_{v_s}J.
  \end{align*}
  Therefore we are free to choose for $m_s$ any element of $R/P^{k_s}R$ which satisfies
  \begin{equation*}
    v(m'_s)\geq \max(\partial_{k_s}L, \partial_{v_s}J).
  \end{equation*}
  Thus the number of possible choices of $m'_s$ of any given valuation is a polynomial in $q$ with coefficients that are integers which do not depend on $R$.
  Having fixed $m'_s$, we are free to choose $m'_{s-1}$ satisfying
  \begin{gather*}
    v(m'_{s-1})\geq \partial_{k_{s-1}}L\\
    v(m'_{s-1}+\pi^{v_{s-1}-v_s}m'_s) \geq \partial_{v_{s-1}+k_s-v_s}J.
  \end{gather*}
  Note that for any $x,y \in R/P^kR$ and non-negative integers $u,v$, the cardinality of the set
  \begin{equation*}
    \{x\mid v(x + y)\geq v \text{ and } v(x)=u\}
  \end{equation*}
  is a polynomial in $q$ with coefficients that are integers which do not depend on $R$.
  This shows that for each fixed valuation of $m'_s$, the number of possible choices for $m'_{s-1}$ of a fixed valuation is again a polynomial in $q$ whose coefficients are integers that do not depend on $R$.
  Continuing in this manner, we find that the cardinality of $S$ is a polynomial in $q$ whose coefficients are integers which do not depend on $R$.
\end{proof}
Proceeding exactly as in the proof of the Main Theorem (Theorem~\ref{theorem:main1}) we can obtain the following refinement:
\begin{theorem}
  [Main theorem]
  \label{theorem:main2}
    Let $R$ be a discrete valuation ring with finite residue field of order $q$.
  Fix $\lambda \in \Lambda_0$ and take $M$ as in (\ref{eq:module-form}).
  Let $G$ denote the group of $R$-module automorphisms of $M$.
  Fix order ideals $I,J\in J(P)_\lambda$ (and hence $G$-orbits $M^*_I$ and $M_J^*$ in $M$).
  \begin{enumerate}
  \item The cardinality of each $G$-orbit in $M^*_I\times M^*_J$ is a monic polynomial in $q$ whose coefficients are integers.
  \item Given a monic polynomial $\beta(q)$ with integer coefficients, the number of $G$-orbits in $M^*_I\times M^*_J$ of cardinality $\beta(q)$ is a polynomial in $q$ with coefficients that are integers which do not depend on $R$.
  \end{enumerate}
\end{theorem}
If we are only interested in the number of orbits (and not the number of orbits of a given cardinality), Corollary~\ref{cor:independence of multiplicities} allows us to reduce any $\lambda\in \Lambda$ to $\lambda_2\in \Lambda_0$.
\begin{theorem}
  \label{theorem:general-lambda}
  Let $R$ be a discrete valuation ring with finite residue field of order $q$.
  Fix $\lambda \in \Lambda$ and take $M$ as in (\ref{eq:module-form}).
  Let $G$ denote the group of $R$-module automorphisms of $M$.
  Fix order ideals $I,J\in J(P)_\lambda$ (and hence $G$-orbits $M^*_I$ and $M_J^*$ in $M$).
  Then the number of $G$-orbits in $M_I^*\times M_J^*$ is given by a polynomial with coefficients that are integers which do not depend on $R$.
\end{theorem}
\section{Relation to representations of quivers}
\label{sec:relat-repr-quiv}
Consider the quiver $Q$ represented by
\begin{equation*}
  \begin{tikzpicture}
    [->,shorten >=1pt,auto,node distance=3cm,
  thick,main node/.style={circle,draw}]

  \node[main node] (1) {$1$};
  \node[main node] (2) [right of=1] {$2$};

  \path[every node/.style={font=\sffamily\small}]
    (1) edge [loop left] node {$\tilde A$} (1)
    (2) edge [bend right] node {$\tilde x$} (1)
    (2) edge [bend left] node {$\tilde y$} (1);
  \end{tikzpicture}
\end{equation*}
To an $n\times n$ matrix $A$ and two $n$-vectors $x$ and $y$ (all with coordinates in a finite field $\Fq$), we may associate a representation of this quiver with dimension vector $(n,1)$ by taking $V_1 = \Fq^n$, $V_2=\Fq$, the linear map corresponding to the arrow $\tilde A$ given by $A$, the linear maps corresponding to the arrows $\tilde x$ and $\tilde y$ being those which take the unit in $V_2=\Fq$ to the vectors $x$ and $y$ respectively.
The representations corresponding to triples $(A,x,y)$ and $(A,x',y')$ are isomorphic if and only if there exists an element $g\in GL_n(\Fq)$ such that
\begin{equation*}
  gAg\inv = A,\:gx =x',\: gy = y'.
\end{equation*}
Thus, the isomorphism classes of representations of $Q$ are in bijective correspondence with triples $(A,x,y)$ consisting of an $n\times n$ matrix and two $n$-vectors up to a simultaneous change of basis.

If we view $k^n$ as a $k[t]$-module $M^A$ where $t$ acts via the matrix $A$, then the number of isomorphism classes of representations of the form $(A,x,y)$ with $A$ fixed may be interpreted as the number of $G^A =\Aut_{k[t]}M$-orbits in $M^A\times M^A$.
Suppose that $k$ is a finite field of order $q$.
The total number of isomorphism classes of representations of $Q$ with dimension vector $(n,1)$ is given by
\begin{equation}
  \label{eq:quiver-reps-and-orbits}
  R_{n,1}(q) = \sum_A |G^A\backslash M^A\times M^A|,
\end{equation}
where $A$ runs over a set of representatives for the similarity classes in $M_n(k)$.
This polynomial was introduced by Kac in \cite{kac1983root}, where he asserted that for any quiver, the number of isomorphism classes of representations with a fixed dimension vector is a polynomial in $q$ with integer coefficients.
He conjectured the non-negativity of a related polynomial (which counts the number of isomorphism classes of absolutely indecomposable representations) from which the non-negativity of coefficients of $R_{n,1}(q)$ follows (see Hua \cite{hua}).
Kac's conjecture was proved by Hausel, Letellier and Rodriguez-Villegas \cite{HLRAnnals} recently.

We now explain how the results of this paper (together Green's theory of types of matrices \cite{MR0072878}) enable us to compute the right hand side of (\ref{eq:quiver-reps-and-orbits}).
Let $\irr \Fq[t]$ denote the set of irreducible monic polynomials in $\Fq[t]$.
Let $\Lambda_0$ denote the subset of $\Lambda$ consisting of element $\lambda$ of type (\ref{eq:lambda}) for which all the cardinals $m_i$ are finite (this is just the set of all partitions).
For $\lambda\in \Lambda_0$ as in (\ref{eq:lambda}), let $|\lambda| = \sum m_i\lambda_i$.
Recall that similarity classes of $n\times n$ matrices with entries in $\Fq$ are parametrized by functions:
\begin{equation*}
  c: \irr \Fq[t]\to \Lambda_0
\end{equation*}
such that
\begin{equation*}
  \sum_{f\in \irr \Fq[t]} \deg f |c(f)| = n
\end{equation*}
The above condition imposes the constraint that $c(f)$ is the empty partition $\emptyset$ with $|\emptyset|=0$ for all but finitely many $f\in \irr \Fq[t]$.
The similarity classes parametrized by $c$ and $c'$ are said to be of the same type if there exists a degree-preserving bijection $\sigma:\irr \Fq[t]\to\irr \Fq[t]$ such that $c'=c\circ \sigma$.

Given a function $c:\irr \Fq[t]\to \Lambda$ parametrizing a similarity class of $n\times n$ matrices, let $\tau_c$ denote the multiset  of pairs $(c(f),\deg f)$ as $f$ ranges over the set of irreducible polynomials in $\Fq[t]$ for which $c(f)\neq \emptyset$.
Then $c$ and $c'$ are of the same type if and only if $\tau_c=\tau_{c'}$.
Thus, the set of types of $n\times n$ matrices with entries in $\Fq$ is parametrized by multisets of the form
\begin{equation}
  \label{eq:type-multiset}
  \tau = \{(\lambda^{(1)},d_1)^{a_1}, (\lambda^{(2)},d_2)^{a_2},\dotsc\}
\end{equation}
such that
\begin{equation}
  \label{eq:types}
  \sum_i a_id_i|\lambda^{(i)}| = n.
\end{equation}
Let $T(n)$ denote the set of multisets of pairs in $\Lambda_0\times \mathbf \Z_{>0}$ satisfying (\ref{eq:types}).
For example, for $T(2)$ has four elements given by:
\begin{itemize}
\item $\{((1,1),1)\}$ (central type)
\item $\{((2),1)\}$ (non-semisimple type)
\item $\{((1),2)\}$ (irreducible type)
\item $\{((1),1)^2\}$ (split regular semisimple type)
\end{itemize}

If $A$ is an $n\times n$ matrix of type $\tau$ as in (\ref{eq:type-multiset}), then by primary decomposition,
\begin{equation*}
  n_\tau(q) = |G^A\bsl M^A\times M^A| = \prod_i n_{\lambda^{(i)}}(q^{d_i}),
\end{equation*}
where $n_\lambda(q)$ is denotes the cardinality of $|G_\lambda\backslash M_\lambda\times M_\lambda|$ when the residue field of $R$ has cardinality $q$.
It is also easy to enumerate the number of similarity classes of a given type $\tau$: for each positive integer $d$, let $m_d$ denote the number of times a pair of the form $(\lambda,d)$ occurs in $\tau$ (counted with multiplicity).
Let $\Phi_d(q)$ denote the number of irreducible polynomials in $\Fq[t]$ of degree $d$ in $\Fq[t]$.
This is a polynomial in $q$ with rational coefficients:
\begin{equation*}
  \Phi_d(q) = \frac 1d \sum_{e|d} \mu(d/e)q^e,
\end{equation*}
where $\mu$ is the classical M\"obius function.
The number of similarity classes of type $\tau$ is
\begin{equation*}
  c_\tau(q) = \frac 1{\prod_i a_i!}\prod_d \Phi_d(q)(\Phi_d(q)-1)\dotsb (\Phi_d(q)-m_d+1).
\end{equation*}
We obtain a formula for $R_{n,1}(q)$:
\begin{equation}
  \label{eq:Rn1}
  R_{n,1}(q) = \sum_{\tau \in T(n)} c_\tau(q)n_\tau(q),
\end{equation}
which can also be expressed as a product expansion for the generating function of $R_{n,1}(q)$ in the spirit of Kung \cite{kung1981cycle} and Stong \cite{stong1988some}:
\begin{equation}
  \label{eq:gen-func}
  \sum_{n=0}^\infty R_{n,1}(q)x^n = \prod_{d=1}^\infty \left(\sum_{\lambda\in \Lambda_0} n_\lambda(q^d)x^{d|\lambda|}\right)^{\Phi_d(q)}.
\end{equation}

There is an alternative method for computing $R_{n,1}(q)$, namely the Kac-Stanley formula \cite[Page 90]{kac1983root}, which is based on Burnside's lemma and a theory of types adapted to quivers (this formula is a way to compute the number of isomorphism classes of representations of a quiver with any dimension vector).
A comparison of the values obtained for $R_{n,1}(q)$ using these two substantially different methods verifies the validity of our results.
This has been carried out by computer for values of $n$ up to $18$ (the code for this can be found at \url{http://www.imsc.res.in/~amri/pairs/}).
\bibliographystyle{abbrv}
\bibliography{refs}
\end{document}